\documentclass[11pt]{article}
\usepackage{amsmath,amsthm,amsfonts,amssymb,bm,wasysym}
\usepackage{amsxtra}
\usepackage{amstext}
\usepackage{hyperref}
\usepackage{verbatim}
\usepackage{float}
\usepackage{comment}
\usepackage[english]{babel}

\usepackage{cite}
\usepackage{hyperref}
\hypersetup{colorlinks=true, urlcolor= black, linkcolor=black, citecolor=blue}

\parskip \medskipamount
\newtheorem {lem} {Lemma} [section]
\newtheorem {prop} {Proposition} [section]
\newtheorem {theo} {Theorem} [section]

\newtheorem {rem} {Remark} [section]

\newtheorem {Hypo} {Hypothesis} [section]

\newcommand{\1}{{\text{\Large $\mathfrak 1$}}}

\title{Vertex reinforced branching random walks and generalized time-dependent P\'olya urns}

\author{Bruno Schapira
\footnote{Universit\'e Claude Bernard Lyon 1, Institut Camille Jordan, CNRS UMR 5208, 43 Boulevard du 11 novembre 1918, 69622 Villeurbanne Cedex, France;  schapira@math.univ-lyon1.fr}}

\begin{document}
\maketitle

\begin{abstract} We consider a class of infinite critical tree-indexed random walks on $\mathbb Z$, where the motion of particles is subject to vertex reinforcement. We mainly focus on the strong reinforcement regime, where we expect the process to localize almost surely on two sites. Part of our analysis includes the study of a time-dependent generalized P\'olya urn process, where the number of draws at each step is prescribed by a sequence $(\sigma_n)_{n\ge 1}$ of arbitrary positive integers, and the probability to pick a ball of a given color is proportional to a function of the {\it number} of balls of that color. In particular for bounded sequences $(\sigma_n)_{n\ge 1}$, we recover Rubin's characterization for the fixation of one color.   
%\newline
%\newline
%\emph{Keywords and phrases.} Self-interacting random walks; Vertex reinforced random walks, Branching random walks, P\'olya urn processes, strong reinforcement, localization, fixation.\\
%MSC 2010 \emph{subject classifications.} 60K35.
\end{abstract}

\section{Introduction} 
This paper is concerned with two models of random processes whose evolution is governed by a reinforcement mechanism. The first one is a generalization of the model of vertex reinforced random walks where instead of studying the motion of only one particle, we consider a family of particles that evolve simultaneously in the same landscape, and which split independently at each time step into a mean one number of new particles (and conditioned to survive forever). The second model is a generalized time-dependent P\'olya urn process, where the number of draws at each step is arbitrary and for each  draw the probability of choosing a ball of a given color is proportional to a function of the number of balls of that color in the urn.

These two models are related by the fact that the restriction of the first one to three sites is equivalent to the second one where the number of draws at each step is itself random and given by the number of individuals in odd generations of an infinite critical random tree.

\subsection{Vertex Reinforced Branching random walk on $\mathbb Z$}
Consider a non-degenerate critical Bienaym\'e-Galton-Watson (BGW) tree conditioned to survive forever, also called Kesten's tree, which is made of a semi-infinite line of vertices, called the spine, to which are attached independent (type of) critical BGW trees. Now consider a random walk indexed by such infinite tree, where for each $n\ge 1$, particles at generation $n$ independently jump to a nearest neighbor of their ancestor's position with probability proportional to a function (called here {\it weight function}) of the time spent on the target vertex in the past generations. We call this type of processes \textbf{Vertex reinforced branching random walks} (VRBRW), in analogy with the standard model of vertex reinforced random walk \cite{P,V}, where the only difference is that in the latter case the process is indexed by the set of non-negative integers, i.e.~a degenerate critical BGW tree. 

\vspace{0.2cm}
We mainly focus here on the strong reinforcement regime where the weight function is reciprocally summable, and where we expect that the VRBRW localizes almost surely on two sites (as for a standard VRRW). Our main result, Proposition~\ref{prop.strong} below, partly confirms this prediction, by showing that the process indeed localizes almost surely on a finite number of sites, and on two sites with positive probability.  

\vspace{0.2cm}
Note that while we only study vertex reinforcement in this paper, edge reinforcement could also be considered. However, in view of the standard models, it is expected to undergo a very different behavior, see e.g.~\cite{CT17,Dav,T2}, and as such we feel that it deserves a separate treatment. 
Note furthermore that while the case of critical offspring distribution is the one 
where we expect VRBRW to behave as in the standard setting without branching, the case of non-critical offspring distributions would also deserve further attention, see in particular~\cite{GS} for a recent work on a model of branching random walk on a triangle with edge reinforcement and supercritical offspring distribution.  
We also refer to e.g.~\cite{ER,GMR,PR} for other models of interacting vertex or edge reinforced random walks. 

\vspace{0.2cm}
Let us now define more precisely our model. 
We consider $\mu$ a probability measure on $\mathbb N$ with mean one and finite second moment, and let $\mathcal T_\infty$ denote the associated Kesten's tree (see Section~\ref{sec.Kesten} below for a precise definition). We also let $\mathcal T_n$ denote the restriction of this tree up to generation $n$, and let $\partial \mathcal T_n$ be the set of vertices at generation $n$. 
For a vertex $u\in \mathcal T_\infty$, we denote by $\mathcal V(u)$ the set of children of $u$. 
Furthermore, we consider a so-called \textbf{weight function}, which is a function $w:\mathbb N\to (0,\infty)$, that is assumed in the whole paper to be {\it non-decreasing}. We then define the VRBRW $(X_u)_{u\in \mathcal T_\infty}$ on $\mathbb Z$ as follows (note that this could in fact be defined as well on any locally finite graph). 
First the local times are defined classically as follows: for any $x\in \mathbb Z$, and any $n\in \mathbb N\cup \{\infty\}$,  
$$Z_n(x) := \sum_{u\in \mathcal T_n} \1\{X_u = x\}. $$ 
By convention we assume that $X_\emptyset = 0$, where we denote by $\emptyset$ the root of $\mathcal T$, so that the initial local time is given by 
$$Z_0(x) = \1\{x=0\},\quad \text{for }x\in \mathbb Z.$$ 
For $n\ge 0$, we let $\mathcal F_n = \sigma(\mathcal T_n, (X_u)_{u\in \mathcal T_n})$.    
Then for each $n$, any $u\in \partial \mathcal T_n$, any $v\in \mathcal V(u)$, and any $x\in \mathbb Z$, on the event $\{X_u = x\}$, 
\begin{equation}\label{def.process}
\mathbb P(X_v = x\pm 1 \mid  \mathcal F_n, \mathcal V(u)) = \frac{w(Z_n(x\pm 1))}{w(Z_n(x-1)) + w(Z_n(x+1))}.
\end{equation}
The main quantities of interest for us will be  
$$\mathcal R:= \{x \in \mathbb Z : Z_\infty(x) >0\} \quad \text{and}\quad \mathcal R':=\{x\in \mathbb Z : Z_\infty(x) = \infty\},$$
respectively the set of sites ever visited and those visited infinitely often. We shall say here that localization occurs, when $\mathcal R$ is a finite set. Regarding known results in the classical setting, we expect that when the sequence $(w(n))_{n\ge 0}$ is reciprocally summable, the process should localize almost surely on two sites, see~\cite{V,BSS}. Our main result regarding VRBRWs is a first step towards this goal.   

\begin{prop} \label{prop.strong}
If $\sum_{n\ge 0} \tfrac 1{w(n)}<\infty$, then, 
$$\mathbb P(|\mathcal R|<\infty) = 1 \quad \text{and}\quad \mathbb P(|\mathcal R'|=2)>0.$$ 
\end{prop}

The case of a non reciprocally summable weight sequence is expected to be much more involved, as regarding the classical setting, one expect localization on different patterns, or no localization, depending on the strength of the reinforcement, see~\cite{BSS,BSS2,CK,Sch12,Sch21,S,T1,V}. In our case, our understanding is rather fragmentary at the moment,  for we can only show that localization on two or three sites is not possible.  
\begin{prop}\label{prop.3sites}
If $\sum_{n\ge 0} \tfrac 1{w(n)}=\infty$, then $\mathbb P(|\mathcal R'|\in \{2,3\})=0$. 
\end{prop}

\subsection{Generalized time-dependent P\'olya urn process} 
The second model we consider in this paper is a generalized time-dependent P\'olya urn process, where at each step one is allowed to pick more than one ball from the urn. This model is linked with the previous one by 
the fact that the restriction of a VRBRW to three sites is equivalent to an urn of this type (see Section~\ref{sec:VRBRW3sites}), where the number of draws at each step is itself random and given by the number of vertices in Kesten's tree at odd generations. 

\vspace{0,2cm}
Similar urn models have been already considered in the past~\cite{RT24,Si18a,Si18b}, but with a different rule for the probabilities of drawing balls; namely in these papers the probability of picking a ball of a specific color is proportional to some function (called feedback) of the {\em proportion} of balls of that color in the urn, while in our setting it will be proportional to a function of the {\em number} of balls of that color. The case with more than two colors and a general balanced replacement matrix has also been considered, see e.g.~\cite{KS,KM17,KMP,LMS,MS25} and references therein.  

\vspace{0,2cm}
We now define more precisely our model. Fix $w:\mathbb N \to (0,\infty)$ a non-decreasing function and $(\sigma_n)_{n\ge 1}$ some sequence of positive integers. 
Start with an urn with $R_0$ red balls and $B_0$ blue balls.  
For each $n\ge 1$, step $n$ consists in adding instantaneously to the urn $\sigma_n$ new balls, whose colors are chosen with probability proportional to the weight of the number of balls of each color 
in the urn, independently for each ball.  
In particular after $n$ steps the number of balls in the urn is equal to
$$\tau_n := R_0 + B_0 + \sigma_0+\dots + \sigma_n,$$ 
and if we denote by $R_n$ and $B_n$ respectively the number of red and blue balls after $n$ steps, then for each $n\ge 0$, conditionally on $R_n$, one has the equality in law
$$R_{n+1} = R_n + \mathcal B\Big( \sigma_{n+1}, \frac{w(R_n)}{w(R_n) + w(B_n)}\Big),$$
where we denote by $\mathcal B(N,p)$ a Binomial random variable with parameters $N\ge 1$ and $p\in [0,1]$. Moreover, for each $n\ge 0$, by definition $R_n + B_n = \tau_n$. 
We also define classically $R_\infty =\lim_{n\uparrow \infty} R_n$ and $B_\infty = \lim_{n\uparrow \infty} B_n$. 

\vspace{0.2cm}
A major question for this kind of processes is to obtain criteria for fixation of one color, i.e.~conditions on $w$ and $(\sigma_n)_{n\ge 1}$, ensuring that either $R_\infty$ or $B_\infty$ is finite almost surely. 
In the original setting where $\sigma_n = 1$ for all $n\ge 1$, 
the process was introduced and studied in~\cite{Dav}, where Rubin's algorithm is presented, and shows that fixation occurs with probability $1$ or $0$, depending on whether 
$\sum_{n\ge 0} \frac 1{w(n)}$ is finite 
or infinite. In another direction, when the weight function is of the form $w(n) = n^\alpha$, for some $\alpha>0$, then our 
process is equivalent to the one studied in~\cite{RT24}, where colors are chosen according to the weight (called feedback there) of the {\it proportion} of balls of that color. In particular when $\alpha>1$, we deduce from~\cite{RT24} that fixation occurs almost surely, if $\sum_{n\ge 1} \frac{\sigma_n}{\tau_n} = \infty$ and 
$\sum_{n\ge 1} (\frac{\sigma_n}{\tau_n})^2 <\infty$. Note that the observation that the two processes are the same for polynomial weights had already been observed and used in~\cite{BRS}. 
Let us stress however, that when $w(n)$ is not a fixed power of $n$, one cannot (a priori) rely on stochastic approximation techniques, as in~\cite{RT24}. 

\vspace{0.2cm}
Our main result extends Rubin's criterium for fixation to the case where $(\sigma_n)_{n\ge 0}$, is any bounded sequence. 

\begin{theo} \label{thm:GPU}
Assume that $(\sigma_n)_{n\ge 1}$ is bounded. Then 
$$\sum_{k\ge 1} \frac 1{w(k)}<\infty\ \Longleftrightarrow\ \mathbb P(R_\infty <\infty) >0\ \Longleftrightarrow\ \mathbb P(R_\infty <\infty\textrm{ or }B_\infty<\infty) = 1. $$ 
\end{theo}
The proof is based on a continuous time embedding of the process, which extends Rubin's algorithm to arbitrary sequences $(\sigma_n)_{n\ge 1}$, see Section~\ref{sec:GPU2}. The difficulty here is that the time-lines associated to each color are not independent, and even worst they are not monotone functions of the individual time clocks associated to each ball. However, when $(\sigma_n)_{n\ge 1}$ is bounded, and $w$ is reciprocally summable, one can still observe that the total time of the process is finite and cannot be the same for both colors, almost surely, which proves that fixation occurs. 

\vspace{0.2cm}
In another direction, we also show that if $w$ grows sufficiently fast, e.g.~stretched exponentially fast, and $(\sigma_n)_{n\ge 1}$ at least polynomially fast, with a sufficiently large exponent, then fixation occurs almost surely. More precisely, we prove the following. 

\begin{theo}\label{thm:GPUexp}
Assume that $w(n) = \exp(cn^\alpha)$, and $\sigma_n = n^\beta$, for some positive constants $c$, $\alpha$ and $\beta> \frac{1-\alpha}{\alpha}$. Then 
$$  \mathbb P(R_\infty <\infty\textrm{ or }B_\infty<\infty) = 1. $$ 
\end{theo}
In fact we prove that almost sure fixation holds under some more general conditions, see our Hypotheses~\ref{hypo} and~\ref{hypo2}, and the corresponding Theorems~\ref{theo.GPUexpbis} and~\ref{thm.loc.bis}.  
However, it remains a challenge to extend these results to more general weight functions $w$ and sequences $(\sigma_n)_{n\ge 1}$, notably for $w$ growing only polynomially fast. 
Actually, as we observe in Remark~\ref{Rem:weak01}, there are some reasons to believe that localization should occur almost surely as soon as 
$$\sum_{n\ge 0} \frac{\sigma_{n+1}}{w(\lceil \tau_n/2\rceil)} <\infty,$$
since under this condition, for the continuous-time embedding of the process mentioned above, the time-line is almost surely finite.

As an application we prove that if for instance $w(n) = \exp(cn)$, for some constant $c>0$, then a VRBRW restricted to three sites localizes almost surely on two sites, see Proposition~\ref{prop:appVRBRW}. 
It remains a challenge to extend this result to the VRBRW on $\mathbb Z$, even for such very strong reinforcement. 

\vspace{0.2cm}
\textbf{Organization:} In the next section we collect some general facts about Kesten's tree. Section~\ref{sec:3sites} is devoted to the proof of Proposition~\ref{prop.3sites}, while Section~\ref{sec:urns} is devoted to the analysis of our urn model, and the application to the VRBRW restricted to three sites. 
Finally in Section~\ref{sec:propstrong} we give a proof of Proposition~\ref{prop.strong}. 

%%%%%%%%%%%%%%%%%%%%%%%%%%%%%%%%%%%%%%%%%%%%%%%%%%%%%%%%%%%%%%%%%%%%

\section{Preliminaries on Kesten's tree}\label{sec.Kesten}
Consider $\mu$ a probability measure on $\mathbb N$ with mean one, and variance $\sigma^2= \sum_{i\ge 1} i(i-1)\mu(i)\in (0,\infty)$. Then we define Kesten's tree $\mathcal T_\infty$ as follows. First, the vertices of this tree are of two types, either normal or special. Normal vertices produce a number of offspring according to the distribution $\mu$, while special vertices produce a number of offspring according to the size biased distribution $\widetilde \mu$, defined by $\widetilde \mu(i ) = i\mu(i)$. Among the children of special vertices, one is chosen uniformly at random to be special, while others (if they exist) are normal. Moreover, the root is a special vertex. Hence the set of special vertices forms an infinite line in $\mathcal T_\infty$, which is also called the spine.

For $u\in \mathcal T_\infty$, we let $|u|$ be the generation of $u$. Recall that we definte $\mathcal T_n = \{u\in \mathcal T_\infty : |u|\le n\}$ and $\partial \mathcal T_n = \{u\in \mathcal T_\infty : |u|=  n\}$. We further set 
$$\mathcal Z_n= |\partial \mathcal T_n|,$$  
the number of vertices of the tree at generation $n$.  In particular by definition, one has $\mathcal Z_0 = 1$, since the root of the tree is the only vertex at generation $0$. We shall need the following elementary and standard results. 
\begin{lem}\label{lem:esp}
One has for any $1\le m \le n$, 
$$ \mathbb E[\mathcal Z_n\mid \mathcal Z_m] = \mathcal Z_m +  (n-m)\sigma^2. $$
\end{lem}
\begin{proof}
The result is immediate by induction, using that at each generation there is exactly one special vertex, whose offspring distribution has mean $1+\sigma^2$, and all other vertices have a number of offspring with mean one. 
\end{proof}

Another result we shall need is the following one, interesting on its own. Note that we just assume that $\mu$ has a finite second moment here.  
This result will only be used for proving Proposition~\ref{prop:appVRBRW}, hence its proof can be skipped at first reading. 

\begin{lem}\label{lem:ratiolim}
One has almost surely, 
$$\lim_{n\to \infty} \frac{\mathcal Z_n}{|\mathcal T_n|} = 0. $$
\end{lem}
\begin{proof}
For $n\ge 0$, denote by $X_n$ the number of normal children of the $n$-th vertex on the spine. In particular $\mathbb E[X_n] = \sigma^2$, for all $n\ge 0$, and let $\mathcal X = \sigma((X_n)_{n\ge 0})$, denote the sigma-field generated by all these variables. For $k\le n-1$, denote by $\mathcal Z^{(k)}_{n-k}$ the number of 
descendent at generation $n$, of one of the normal children of the $k$-th vertex on the spine. In particular $\mathbb E[\mathcal Z^{(k)}_{n-k} \mid X_k] = X_k$, for all $k\le n-1$. Recall furthermore, that if $\widetilde {\mathcal Z}_i$ denote the number of individuals at $i$-th generation in a BGW tree with offspring distribution $\mu$, then by an immediate induction, it can be shown that $\text{Var}( \widetilde {\mathcal Z}_i) = i\sigma^2$, for all $i\ge 0$. Thus, by observing that conditionally on $X_k$, 
$\mathcal Z^{(k)}_{n-k}$ is a sum of $X_k$ independent random variables distributed as $ \widetilde {\mathcal Z}_{n-k-1}$, one can deduce that
$$\text{Var}(\mathcal Z^{(k)}_{n-k}\mid X_k) = (n-k-1)\sigma^2 X_k.$$
Since by definition, 
$$\mathcal Z_n = 1+ \sum_{k=0}^{n-1} \mathcal Z^{(k)}_{n-k}, $$ 
we deduce that 
$$\mathbb E[\mathcal Z_n\mid \mathcal X] = 1+ \sum_{k=0}^{n-1} X_k, \quad \text{and}\quad \text{Var}(\mathcal Z_n\mid \mathcal X) \le n \sigma^2\sum_{k=0}^{n-1}X_k. $$ 
Chebychev's inequality then gives 
$$\mathbb P\Big(\mathcal Z_n \ge 1+ \big(\sum_{k=0}^{n-1} X_k\big) + n^{3/2} \log n\ \Big|\ \mathcal X \Big)\le \frac{n \sigma^2 \sum_{k=0}^{n-1} X_k}{n^3(\log n)^2}. $$ 
Now the strong law of large numbers entails that almost surely, as $n\to \infty$, 
$$\frac 1n \sum_{k=0}^{n-1} X_k \to \sigma^2,$$
which implies that almost surely $\sum_{k=0}^{n-1} X_k \le 2n\sigma^2$, for all $n$ large enough. 
Hence almost surely, 
$$\sum_{n\ge 0} \mathbb P\Big(\mathcal Z_n \ge 1+ \big(\sum_{k=0}^{n-1} X_k\big) + n^{3/2}\log n\ \Big|\ \mathcal X \Big)<\infty, $$ 
which in turn implies using Borel-Cantelli lemma, that almost surely $\mathcal Z_n \le 2n^{3/2}\log n$, for all $n$ large enough.

We now proceed with a lower bound for $|\mathcal T_n|$. For every $0\le k\le n-1$, and any $1\le i\le X_k$, denote by $\mathcal T^{(k,i)}_n$ the tree of descendent of the $i$-th
normal child of the $k$-th vertex on the spine, up to generation $n$, so that by adding the $(n+1)$ vertices on the spine we obtain
$$|\mathcal T_n| =n+1+ \sum_{k=0}^{n-1} \sum_{i=1}^{X_k} |\mathcal T^{(k,i)}_n|. $$ 
Let $N = \lfloor n^{4/5}\rfloor$,  and for $1\le j\le n^{1/5}$,  
$$U_j = \sum_{k = jN}^{(j+1)N - 1}  \sum_{i = 1}^{X_k}  |\mathcal T^{(n-k,i)}_n|. $$ 
Note first that, 
$$\mathbb E[U_j\mid \mathcal X] =  \sum_{k = jN}^{(j+1)N - 1} kX_k\ge jN\sum_{k=jN}^{(j+1)N - 1} X_k.$$
Furthermore, a straightforward computation shows that $\text{Var}(|\mathcal T^{(k,i)}_n|) \le C(n-k)^3$, for some constant $C>0$, which entails 
$$\text{Var}(U_j \mid \mathcal X) \le C(j+1)^3N^3 \sum_{k = jN}^{(j+1)N - 1} X_k. $$  
Hence on the event $\big\{ \sum_{k = jN}^{(j+1)N - 1} X_k \ge N \sigma^2 / 2\big\}$, we obtain using Paley-Zygmund's inequality, 
$$\mathbb P\Big(U_j \ge \frac{jN^2\sigma^2}{4} \mid \mathcal X\Big)  \ge \frac 14 \cdot \frac{ \mathbb E[U_j\mid \mathcal X]^2}{ \mathbb E[U_j\mid \mathcal X]^2 + 
\text{Var}(U_j\mid \mathcal X)} \ge \frac{c}{j}, $$
for some constant $c>0$. Integrating with respect to $\mathcal X$, and using that by Chebychev's inequality $\mathbb P(\sum_{k = jN}^{(j+1)N - 1} X_k \ge N \sigma^2 / 2) \ge c'$, for some $c'>0$, we get 
$$\mathbb P\big(U_j \ge \frac{jN^2\sigma^2}{4}\big)\ge \frac{cc'}{j}. $$ 
Since the $(U_j)_{1\le j\le N}$ are independent, we get that for some $\alpha>0$, and for $n$ large enough, 
$$\mathbb P\big(|\mathcal T_n|\ge \frac{N^2\sigma^2}4\big)\ge  \mathbb P\Big(\exists j\le N : U_j \ge \frac{jN^2\sigma^2}{4}\Big) \ge 1 - \prod_{j=1}^N (1- \frac{cc'}{j}) \ge 1 - N^{-\alpha}. $$ 
In particular, for $\beta>5/(4\alpha)$, one has 
$$\sum_{n\ge 1} \mathbb P\big(|\mathcal T_{n^\beta}| \leq \frac{n^{8\beta/5} \sigma^2}{5}\big)<\infty. $$ 
It then follows from Borel-Cantelli lemma, that almost surely for all $n$ large enough, 
$|\mathcal T_{n^\beta}|\ge n^{8\beta/5} \sigma^2/5$, and since $(|\mathcal T_n|)_{n\ge 1}$ is monotone, we get that in fact almost surely, for all $n$ large enough, 
$|\mathcal T_n|\ge n^{8/5} \sigma^2/6$. Together with our previous upper bound on $\mathcal Z_n$, and using that $8/5 > 3/2$, this concludes the proof of the lemma. 
\end{proof}

%%%%%%%%%%%%%%%%%%%%%%%%%%%%%%%%%%%%%%%%%%%%%%%%%%%%%%%%%%%%%%%%%%%%

\section{Proof of Proposition~\ref{prop.3sites}} \label{sec:3sites}
We start by defining some useful processes, which are analogues to those introduced by Tarr\`es in the setting of usual VRRWs~\cite{T1,T2}. 
For $x\in \mathbb Z$, and $n\ge 0$, we set  
$$Y_n^\pm(x) := \sum_{u\in \mathcal T_n} \sum_{v\in\mathcal V(u)} \frac{\1\{X_u = x, X_v = x\pm 1\}}{w(Z_n(x\pm 1))}.$$
A first observation is the following. 
\begin{lem}\label{lem:martingale}
Define the process $(M_n)_{n\ge 0}$ by $M_0=0$, and for $n\ge 0$, 
$$M_{n+1}(x) := Y_n^+(x) - Y_n^-(x).$$
Then $(M_n)_{n\ge 0}$ is an $(\mathcal F_n)_{n\ge 0}$-martingale. 
\end{lem}
\begin{proof} The proof follows directly from~\eqref{def.process}. 
\end{proof} 
The next lemma is an analogue of~\cite[Corollary 2.1]{T2}. 

\begin{lem}\label{lem:Ypm}
Assume that $\sum_{n\ge 0} \tfrac 1{w(n)}=\infty$. Then almost surely, for any $x\in \mathbb Z$, 
$$\{Z_\infty(x)=\infty\}\ \subseteq\  \{Y_\infty^+(x-1)=\infty\} \cup \{ Y_\infty^-(x+1)=\infty\}.$$ 
\end{lem}
\begin{proof} By definition one has, 
\begin{align*}
 Y_\infty^+(x-1) + Y_\infty^-(x+1) & = \sum_{n=0}^\infty \frac{Z_{n+1}(x) - Z_n(x)}{w(Z_n(x))}\ge \sum_{n=0}^\infty \sum_{k=Z_n(x)}^{Z_{n+1}(x)-1} \frac 1{w(k)} \\
 & = \sum_{k=Z_0(x)}^{Z_\infty(x)-1}\frac 1{w(k)},
 \end{align*} 
using that $w$ is non-decreasing for the first inequality. This proves the lemma.
\end{proof}

We can now give the proof of Proposition~\ref{prop.3sites}. 

\begin{proof}[Proof of Proposition~\ref{prop.3sites}]
We start by showing that localization on two sites is not possible. For this recall that a martingale with bounded increments almost surely  either diverges towards $+\infty$ or $-\infty$, or oscillates between both, in the sense that its $\liminf$ is $-\infty$ and its $\limsup$ is $+\infty$. Now on the event that $Z_\infty(x) = \infty$ for some $x$, either $Y_\infty^+(x-1)=\infty$ or $Y_\infty^-(x+1) = +\infty$, by Lemma~\ref{lem:Ypm}. In the first case, we deduce using Lemma~\ref{lem:martingale} and the aforementioned result that one also has $Y_\infty^-(x-1) = \infty$, which entails $Z_\infty(x-2) =Z_\infty(x-1)= \infty$. Likewise, if $Y_\infty^-(x+1) = \infty$, we deduce that $Z_\infty(x+1) = Z_\infty(x+2) = \infty$. In any case, we deduce that localization on two sites is not possible.

A similar argument shows that localization on three sites is not possible. Indeed, assume that $Z_\infty(x) = Z_\infty(x-1) = Z_\infty(x+1)$, for some $x\in \mathbb Z$. Then either $Y_\infty^+(x-1)$ or $Y_\infty^-(x+1)$ is infinite. In the first case one also has $Y_\infty^-(x-1) = \infty$ which entails $Z_\infty(x-2) = \infty$, while in the former case we conclude similarly that $Z_\infty(x+2) = \infty$. In both cases, at least $4$ sites are visited infinitely often.   
\end{proof}

%%%%%%%%%%%%%%%%%%%%%%%%%%%%%%%%%%%%%%%%%%%%%%%%%%%%%%

\section{Time-dependent generalized P\'olya urn processes} \label{sec:urns}
We consider now the model of time-dependent generalized P\'olya urn process (GPU), presented in the introduction. Our main results concerning these processes are gathered in 
Sections~\ref{sec:GPU},~\ref{sec:GPU2}, and~\ref{sec:GPU3}, while Section~\ref{sec:VRBRW3sites} provides some applications to the VRBRW restricted to three sites.

\subsection{General facts} \label{sec:GPU}
Our first result collects some simple facts about GPU processes.

\begin{prop} \label{prop:GPU}
Consider a GPU process with weight function $w$ and drawing sequence $(\sigma_n)_{n\ge 1}$. Assume that there are initially $B_0\ge 0$ blue balls. Then 
\begin{itemize}
\item[(i)]
$$\sum_{n\ge 1} \frac {\sigma_n}{w(\tau_{n-1}-B_0)} <\infty   \quad  \Longleftrightarrow \quad \mathbb P(B_\infty= B_0 ) >0,$$ 
\item[(ii)] 
$$ \sum_{n\ge 1} \frac {\sigma_n}{w(\tau_{n-1})} =\infty   \quad  \Longrightarrow \quad \mathbb P(R_\infty = B_\infty = \infty) =1,$$ 
\item[(iii)] 
$$\mathbb P(R_\infty = B_\infty = \infty) =0 \quad  \Longrightarrow \quad \sum_{n\ge b+1} \frac {\sigma_n}{w(\tau_{n-1}-b)} <\infty, \ \text{for all }b>0.$$ 
\end{itemize}
\end{prop}
\begin{proof}
A direct computation shows that the probability of drawing only red balls is equal to 
$$\mathbb P(B_n = B_0, \ \text{for all }n\ge 1) = \prod_{n\ge 1} \big(1- \frac{w(B_0)}{w(B_0) + w(\tau_{n-1}-B_0)}\big)^{\sigma_n},$$
and the right-hand side is positive if, and only if, $\sum_{n\ge 1} \frac{\sigma_n}{w(\tau_{n-1}-B_0)}<\infty$, proving already Part (i).

For (ii), note that if say $\mathbb P(B_\infty<\infty)>0$, then there must exist $k\ge 1$ and $b\ge 0$, such that 
$\mathbb P(B_\infty = B_k =b  ) >0$. By the previous computation and the fact that $w$ is non-decreasing this entails 
$$\sum_{n\ge k+1} \frac {\sigma_n}{w(\tau_{n-1})} \le \sum_{n\ge k+1} \frac {\sigma_n}{w(\tau_{n-1}-b)} <\infty,$$ 
and the same argument applies if $\mathbb P(R_\infty<\infty)>0$.

Finally for (iii), assume that for some $b\ge 0$, one has $\sum_{n\ge b+1} \frac {\sigma_n}{w(\tau_{n-1}-b)} =\infty$. Note that with positive probability, after $2b$ steps there are at least $b$ blue balls and $b$ red balls in the urn (since $\sigma_k\ge 1$ for all $k\ge 1$, by hypothesis). Then the argument given above for proving (ii) shows that almost surely after this, one has $R_\infty = B_\infty = \infty$, concluding the proof of (iii). 
\end{proof}

\begin{rem} 
\em{The previous result shows that if the function $w$ and the sequence $(\sigma_n)_{n\ge 1}$ are such that $\sum_{n\ge 1} \frac {\sigma_{n+1}}{w(\sigma_1+\dots+\sigma_n)} <\infty$, and $\sum_{n\ge 1} \frac {\sigma_{n+1}}{w(\sigma_1+\dots+\sigma_n-b)} =\infty$, for some $b>0$ (see Remark~\ref{rem:ex} below for an example), then for any $R_0,B_0\ge 1$, for the urn starting from $R_0$ red balls and $B_0$ blue balls, one has $\mathbb P(R_\infty = B_\infty = \infty) \in (0,1)$. 
This differs from the situation for usual generalized P\'olya urn processes (corresponding to the case when $\sigma_n = 1$ for all $n\ge 1$), where by Rubin's construction~\cite{Dav}  
\begin{equation}\label{Rubin}
\sum_{k\ge 1} \frac 1{w(k)} <\infty \ \Longleftrightarrow \ \mathbb P(R_\infty <\infty ) >0 \ \Longleftrightarrow \ \mathbb P( R_\infty <\infty \text{ or }B_\infty <\infty) = 1.
\end{equation}
In particular, in this standard setting, one has for any weight function, 
\begin{equation}\label{01law}
\mathbb P(R_\infty = B_\infty = \infty) \in \{0,1\}.
\end{equation}
}
\end{rem}

\begin{rem}\label{rem:ex}\em{
Consider the weight function $w$ defined by $w(k) = e^{e^k}$, for all $k\ge 0$, and define recursively the sequence $(\sigma_k)_{k\ge 1}$ by $\sigma_k = 
w(\sigma_1+\dots+\sigma_{k-1})$. Assume now that $B_0 = R_0 = 1$, so that $\tau_0=2$, and for $k\ge 1$, $\tau_k = \sigma_1+\dots+\sigma_k +  2$. Then note that one has $\sum_{k\ge 1} \frac{\sigma_k}{w(\tau_{k-1}-1)}<\infty$, which by the first part of Proposition~\ref{prop:GPU} implies that $\mathbb P(B_\infty = B_0 = 1) >0$. However, one can also notice that  $\sum_{k\ge 1} \frac{\sigma_k}{w(\tau_{k-1}-2)}=\infty$. Consequently, 
if after the two first steps, we pick at least one red ball and one blue ball (which happens with positive probability, since $\sigma_1 +\sigma_2\ge 2$), then by the second part of Proposition~\ref{prop:GPU} (see also its proof), almost surely after this, both the red and blue balls will be drawn infinitely often. Hence for this process one has $\mathbb P(R_\infty = B_\infty = \infty) \in (0,1)$, in contrast with the situation in the standard setting, recall~\eqref{01law}. Moreover, if we assume now that $B_0 = R_0=0$, then almost surely $R_\infty = B_\infty = \infty$; hence the behavior of a GPU process may be sensitive to the initial conditions. }
\end{rem}

An important question is now to characterize the set of weight functions $w$ and sequences $(\sigma_n)_{n\ge 1}$ for which there is almost surely fixation of one color (i.e.~for which almost surely one color is drawn only finitely many times). 
A natural guess could be that the reverse implication in Proposition~\ref{prop:GPU} (iii) should be true, but this seems a challenging problem. In Remark~\ref{Rem:weak01} below we suggest 
a slightly weaker criterium, that may be  simpler to prove.   

\subsection{Case $(\sigma_n)_{n\ge 1}$ bounded}  \label{sec:GPU2}
Here we prove Theorem~\ref{thm:GPU}, concerning the case when $(\sigma_n)_{n\ge 1}$ is bounded. 
The proof is based on a continuous-time embedding of the process, which generalizes Rubin's algorithm from~\cite{Dav} to arbitrary sequences $(\sigma_n)_{n\ge 1}$. Hence we first define properly 
this algorithm, together with some basic properties (Lemmas~\ref{lem:Rubin} and~\ref{lem.Rubin2}), and then we will give the proof of Theorem~\ref{thm:GPU}. 

\vspace{0.2cm}
\noindent \textbf{Rubin's type algorithm.} 
Let $R_0$ and $B_0$ be given, as well as two arbitrary sequences of positive numbers $(\xi_i^R)_{i\ge 1}$ and $(\xi_i^B)_{i\ge 1}$. Then for $t\ge 0$ (and up to some time $T_1$ defined below), we let  
$$N_t^R = \sum_{i\ge 1} \mathbf 1\big\{\frac{\xi_1^R + \dots + \xi_i^R}{w(R_0)} \le t\big\}, \quad \text{and}\quad N_t^B = \sum_{i\ge 1} \mathbf 1\big\{\frac{\xi_1^B + \dots + \xi_i^B}{w(B_0)} \le t\big\},$$
which represent respectively the number of red and blue balls drawn up to time $t$, during the first step.  
We also define  
$$T_1 = \inf \{t\ge 0 : N_t^R + N_t^B\ge  \sigma_1\},$$
the time duration of the first step, i.e.~the time at which the $\sigma_1$-th ball is drawn. 
Note that it may be that two balls, one of each color, are drawn at this time (in particular it could be that $N_{T_1}^R + N_{T_1}^B =\sigma_1 +1$). This is not an issue, as on one hand the process is still well-defined in this case and moreover, when the sequences $(\xi_i^R)_{i\ge 1}$ and $(\xi_i^B)_{i\ge 1}$ are distributed as mean one exponential random variables, this almost surely never happens (see Lemma~\ref{lem:Rubin} below). We then set 
$$R_1 = R_0 + N_{T_1}^R, \quad  \text{and} \quad B_1 =B_0 +  N_{T_1}^B,$$
for respectively the total number of red and blue balls in the urn after the first step. 
Additionally, we define
\begin{equation}
\widetilde \xi_1 =\left\{ 
\begin{array}{ll}
\Big( \sum_{i=1}^{N_{T_1}^R+1} \xi_i^R\Big) - w(R_0)\cdot T_1 & \text{if } N_{T_1-}^R=N_{T_1}^R,\\
\Big( \sum_{i=1}^{N_{T_1}^B+1} \xi_i^B\Big) - w(B_0)\cdot T_1 & \text{else}.
\end{array}
\right. 
\end{equation}
This random variable represents the excess of time which remains for the clock that has been engaged before time $T_1$ but did not ring yet at that time (unless two clocks ring at time $T_1$, in which case $\widetilde \xi_1=\xi^B_{N_{T_1}^B +1}$ is still well-defined). 

\vspace{0.2cm}
Now suppose that the times $T_k$, for $k\le n$, and the process $(N_t^R,N_t^B)$, for $t\le T_n$, have been defined, as well as $(R_k,B_k)$ and $\widetilde \xi_k$, for $k\le n$. Suppose also to fix ideas that a red ball is drawn at time $T_n$, i.e.~that $N_{T_n-}^R<N_{T_n}^R$ (a similar construction can be made if instead $N_{T_n-}^B<N_{T_n}^B$). Then for $t>T_n$ (up to some time $T_{n+1}$ defined below), we let 
$$N_t^R = N_{T_n}^R +  \sum_{i\ge 1} \mathbf 1\Big\{\frac{\xi_{R_n+1}^R + \dots + \xi_{R_n+i}^R}{w(R_n)} \le t-T_n \Big\}, $$
and
$$N_t^B = N_{T_n}^B +  \sum_{i\ge 1} \mathbf 1\Big\{\frac{\widetilde \xi_n + \xi_{B_n+2}^B+ \dots + \xi_{B_n+i}^B}{w(B_n)} \le t-T_n\Big\}.$$
Furthermore, we let 
$$T_{n+1} = \inf \{t> T_n : N_t^R + N_t^B \ge \sigma_1+\dots+ \sigma_{n+1}\},$$ 
and 
$$R_{n+1} = R_0+ N_{T_{n+1}}^R, \quad B_{n+1} = B_0 + N_{T_{n+1}}^B.$$ 
Finally, if $N_{T_{n+1}-}^R< N_{T_{n+1}}^R$ (i.e.~if at time $T_{n+1}$ a red ball is drawn), then we let 
$$\widetilde \xi_{n+1}= \widetilde \xi_n + \Big(\sum_{i= N_{T_n}^B+2}^{N_{T_{n+1}}^B +1} \xi_i^B\Big)- w(B_n)\cdot (T_{n+1} - T_n),$$ 
and else, let 
$$ \widetilde \xi_{n+1}= \Big(\sum_{i= N_{T_n}^R+1}^{N_{T_{n+1}}^R +1} \xi_i^R\Big)- w(R_n)\cdot (T_{n+1} - T_n).$$ 
By induction, this defines a sequence $(T_n)_{n\ge 1}$, a process $(N_t^R,N_t^B)_{0\le t< T_\infty}$, where $T_\infty = \lim_{n\uparrow \infty} T_n$, as well as a process $(R_n,B_n)_{n\ge 0}$. 
The next lemma states that the above algorithm provides a continuous-time embedding of any GPU process. As anticipated this algorithm provides a continuous time embedding of GPU processes, as stated in the next lemma.  

\begin{lem}\label{lem:Rubin}
Assume that $(\xi_i^R)_{i\ge 1}$ and $(\xi_i^B)_{i\ge 1}$ are distributed as independent mean one exponential random variables. Then the process $(R_n,B_n)_{n\ge 0}$, defined above via Rubin's algorithm,
has the same law as a GPU process starting from $R_0$ red balls and $B_0$ blue balls. In particular, almost surely $N_{T_n}^R + N_{T_n}^B = \sigma_1+ \dots + \sigma_n$, for all $n\ge 1$. 
\end{lem}
\begin{proof}
The proof is standard, but let us recall it for completeness. Denote by $s_k$ the time when the $k$-th ball is drawn. More precisely, let $s_0 = 0$, and define inductively $(s_k)_{k\ge 1}$ by 
$$s_{k+1} = \inf \{t>s_k : N_t^R + N_t^B > N_{s_k}^R + N_{s_k}^B\}. $$ 
Assume that almost surely for any $\ell < k$, one has 
\begin{equation}\label{rec.exp}
N_{s_{\ell +1}}^R + N_{s_{\ell +1}}^B = N_{s_\ell}^R + N_{s_\ell}^B +1,
\end{equation}
and let us prove it now for $\ell = k$. By the induction hypothesis, one has almost surely, $N_{T_n}^R + N_{T_n}^B = \sigma_1+ \dots + \sigma_n$, for all $n$ such that $T_n \le s_k$. 
Letting now 
\begin{equation}\label{def.Ft}
\mathcal F_t= \sigma((N_s^R,N_s^B)_{s\le t}),
\end{equation}
one can observe that conditionally on $\mathcal F_{s_k}$, and on the event $T_n\le s_k <T_{n+1}$, by construction and the memoryless property of exponential random variables, $s_{k+1} - s_k$ is distributed as the minimum between two independent exponential random variables with respective parameters $w(R_n)$ and $w(B_n)$ (i.e.~as an exponential random variable with parameter $w(R_n) + w(B_n)$). In particular, almost surely these two variables are distinct, which proves~\eqref{rec.exp} for $\ell = k$ as desired. Additionally the probability that the $(k+1)$-th ball 
that is picked is red is $w(R_n)/(w(R_n) + w(B_n))$, which by induction proves the second assertion of the lemma. 
\end{proof}
We shall also need the following fact. 
\begin{lem}\label{lem.Rubin2}
Assume that $(\xi_i^R)_{i\ge 1}$ and $(\xi_i^B)_{i\ge 1}$ are distributed as independent mean one exponential random variables, and that $\sum_{n\ge 0} \frac{\sigma_{n+1}}{w(\lceil \tau_n/2\rceil)}<\infty$. 
Then, $T_\infty<+\infty$, almost surely. 
\end{lem}
\begin{proof} Using the well-known fact that if $X$ and $Y$ are distributed as independent exponential random variables with mean $\lambda$ and $\mu$ respectively, then $\min(X,Y)$ follows an exponential law with mean $\lambda + \mu$, we deduce that for each $n\ge 0$, conditionally on $(R_n,B_n)$, almost surely $T_{n+1} - T_n$ is distributed as a sum of $\sigma_{n+1}$ independent exponential random variables with mean $\frac 1{w(R_n) + w(B_n)}$ (with the convention that $T_0 = 0$). Thus, for any $n\ge 0$, almost surely 
$$\mathbb E[T_{n+1}  - T_n \mid \mathcal F_{T_n}] = \frac {\sigma_{n+1}}{w(R_n) + w(B_n)},$$
where $(\mathcal F_t)_{t\ge 0}$ is defined in~\eqref{def.Ft}.  
Moreover,  Lemma~\ref{lem:Rubin} entails that almost surely $R_n + B_n = \tau_n$, and since $w$ is nondecreasing, it follows that almost surely 
$$\mathbb E[T_{n+1}  - T_n \mid \mathcal F_{T_n}] \le  \frac {\sigma_{n+1}}{w(\lceil \tau_n/2\rceil)}. $$ 
Hence the hypothesis of the lemma implies that 
$$\mathbb E[T_\infty] = \sum_{n\ge 0} \mathbb E[T_{n+1} - T_n] <\infty, $$
and the desired result follows. 
\end{proof} 

\begin{rem}\label{Rem:weak01} \emph{Given the above result, it is really tempting to believe that the $0-1$ law of Theorem~\ref{thm:GPU} should be true under the more general hypothesis of this lemma. 
However, the lack of monotonicity in Rubin's construction makes it difficult to prove.  }
\end{rem}

We are now ready for the proof of Theorem~\ref{thm:GPU}, concerning bounded sequences $(\sigma_n)_{n\ge 1}$.
 
\begin{proof}[Proof of Theorem~\ref{thm:GPU}]
By assumption there exists an integer $K>0$, such that $\sigma_n\le K$, for all $n\ge 0$. Assume first that $\sum_{n\ge 1} \frac 1{w(n)}<\infty$, and consider Rubin's construction of the process. 
By Lemma~\ref{lem.Rubin2}, one has $T_\infty<\infty$, almost surely. Now for each $n\ge 0$, we let  
$$A_n = \big\{R_{n+K} \ge \frac{\tau_{n+K}}2 + K\big\},$$
and note that almost surely, 
\begin{equation}\label{prob.An}
\mathbb P(A_n\mid  \mathcal F_{T_n})\cdot \mathbf 1_{\{R_n \ge \tau_n/ 2 \}} \ge \frac{1}{2^{K^2}},
\end{equation}
since, for the event $A_n$ to hold, it suffices that only red balls are picked during the first $K$ steps after time $T_n$. Moreover, we claim that for each $n\ge 0$, 
\begin{equation}\label{claimAnBn}
A_n \cap \Big\{\sum_{k \ge \frac{\tau_{n+K}}{2} + K}  \frac{\xi_k^R}{w(k-K)}<\sum_{k \ge\frac{\tau_{n+K}}{2} - K} \frac{\xi_k^B}{w(k-1)}\Big\} \ \subseteq \ \{B_\infty<\infty\}. 
\end{equation}
Indeed, by construction, on the event $A_n$, we have using the monotonicity of $w$, 
$$ \sum_{B_\infty \ge k \ge  \frac{\tau_{n+K}}{2} - K}  \frac{\xi_k^B}{w(k-1)}\le T_\infty- T_{n+K} \le \sum_{R_\infty  \ge k-1 \ge \frac{\tau_{n+K}}{2} + K}  \frac{\xi_k^R}{w(k-K)},$$
and~\eqref{claimAnBn} follows. Consequently, on the event $\{R_n \ge \tau_n/ 2 \}$, one has almost surely, 
\begin{align*}
\mathbb P(B_\infty<\infty\mid\mathcal F_{T_n})  & \ge  \frac{1}{2^{K^2}} \cdot \mathbb P\Big(\sum_{k \ge \frac{\tau_{n+K}}{2} + K}  \frac{\xi_k^R}{w(k-K)}<\sum_{k \ge\frac{\tau_{n+K}}{2} - K} \frac{\xi_k^B}{w(k-1)}\Big) \\ 
&  \ge  \frac{1}{2^{K^2+1}}, 
\end{align*}
where for the last inequality, we use that if $X$ and $X'$ are two independent random variables with the same law, and $Z$ is almost surely positive, then by symmetry, 
$$\mathbb P(X<X' + Z) \ge \mathbb P(X\le X') \ge 1/2.$$  
Similarly, on the event $\{B_n >\tau_n/ 2 \}$, one has almost surely, 
$$\mathbb P(R_\infty<\infty\mid\mathcal F_{T_n}) \ge \frac{1}{2^{K^2+1}}, $$
and therefore, since almost surely either $R_n\ge \tau_n/2$ or $B_n > \tau_n / 2$, one also has 
$$\mathbb P(R_\infty<\infty\text{ or }B_\infty<\infty \mid\mathcal F_{T_n}) \ge \frac{1}{2^{K^2+1}}. $$ 
Letting $n\to \infty$, we deduce from martingale theory that almost surely 
$$\mathbf 1_{\{R_\infty<\infty\text{ or }B_\infty<\infty \}} \ge  \frac{1}{2^{K^2+1}}, $$ 
and thus almost surely only one color is picked infinitely often, as wanted.

Now if $\sum_{n\ge 1} \frac 1{w(n)} = \infty$, it follows from Proposition~\ref{prop:GPU} (ii) that almost surely $R_\infty = B_\infty = \infty$, which completes the proof of the theorem. 
\end{proof}

\subsection{Case of weights growing (stretched) exponentiallly fast}\label{sec:GPU3}
Here we consider another setting, where $w(n)$ grows at least stretched exponentially fast, i.e. like $\exp(n^\alpha)$, with $\alpha>1/2$. In fact our exact hypotheses are as follows. We extend $w$ as a function on $[0,\infty)$ by linear interpolation, and we assume in this whole section that $\sum_{n\ge 0} \frac 1{w(n)}<\infty$ (which is also a consequence of condition $(iii)$ below).  

\begin{Hypo}\label{hypo}
\begin{itemize}
\item[$(i)$] For any $\kappa>0$, 
$$\sum_{n\ge 1} \exp(-\kappa \cdot\sigma_n)<\infty. $$ 
\item[$(ii)$] $$\liminf_{n\to \infty} \frac{w( \frac{\tau_n}{2} + \sqrt{\sigma_n})}{w(\frac{\tau_n}{2})} >1. $$ 
\item[$(iii)$]  For any $\varepsilon>0$, there exists $\delta>0$, such that for any $\gamma\ge 1+\varepsilon$, 
$$\liminf_{x\to \infty} \frac{w(\gamma x)}{w(x)} \ge  \gamma + \delta. $$ 
\item[$(iv)$]  There exists $\varepsilon>0$ and $n_0$, such that for all $n\ge n_0$,  
$$\sum_{i=n}^\infty \frac{\sigma_{i+1}}{w((1-\varepsilon)\tau_i)} \le \frac 14 \sum_{i = \lfloor \varepsilon \tau_n\rfloor }^\infty \frac 1{w(i)}. $$ 
\end{itemize}
\end{Hypo}

\begin{rem}\label{rem:hyp}
\emph{\begin{enumerate}
\item Of course, for $(i)$ to be satisfied, it suffices that $\lim_{n\to \infty} \frac{\sigma_n}{\log n}= \infty$, but note that these two conditions are not equivalent.    
\item Concerning the weight function $w$, the most restrictive hypothesis is the condition $(ii)$, as it imposes some stretched exponential growth, while hypotheses $(iii)$ and $(iv)$ would be satisfied for weights growing only polynomially fast. 
More precisely, 
\begin{itemize}
\item[$\bullet$] if $w(n) = n^\alpha$, with $\alpha>1$, then $(iii)$ is satisfied, and $(iv)$ as well, if $(\sigma_n)_{n\ge 1}$ does not grow too fast, e.g. if $\sigma_n=(\log n)^\beta$, for some $\beta>0$ (in fact one needs $\beta>1$ 
to satisfy also $(i)$), or if $\sigma_n = n^\beta$, with $\alpha> 2$ and $\beta\in (0,\alpha-2]$. 
\item[$\bullet$] On the other hand for $(ii)$ one needs faster growth of the weight function $w$ and $(\sigma_n)_{n\ge 1}$. Indeed, if say $w(n) = \exp(n^\alpha)$, with $\alpha \in (0,1)$ and 
$\sigma_n= n^\beta$, then $(ii)$ is satisfied if and only if $\alpha>1/2$ and $\beta \ge \frac{1-\alpha}{\alpha - \frac 12}$, while if $w(n) = \exp(n^\alpha)$, 
with $\alpha\ge 1$, then $(ii)$ is satisfied for any sequence $(\sigma_n)_{n\ge 1}$.  
\end{itemize}
\end{enumerate}}
\end{rem}

Our next result ensures that fixation occurs almost surely when $w$ and $(\sigma_n)_{n\ge 1}$ satisfy all the conditions of Hypothesis~\ref{hypo}. 
\begin{theo}\label{theo.GPUexpbis} 
Assume that Hypothesis~\ref{hypo} $(i)-(iv)$ hold. Then 
$$\mathbb P(R_\infty<\infty \text{ or } B_\infty<\infty)= 1. $$ 
\end{theo}

Note that based on Remark~\ref{rem:hyp}, one can notice that this result implies that fixation occurs, when $w(n) = \exp(cn^\alpha)$ and $\beta \ge \frac{1-\alpha}{\alpha - \frac 12}$. The extension to arbitrary $\beta \ge \frac{1-\alpha}{\alpha}$, as claimed in Theorem~\ref{thm:GPUexp}, will be discussed later.  
The proof of Theorem~\ref{theo.GPUexpbis} relies on two intermediate results, that we first prove. 

\begin{lem}\label{lem:loc1}
Assume that Hypothesis~\ref{hypo} (i) and (ii) hold. Then there exists $\varepsilon>0$, such that almost surely, 
$$\liminf_{n\to \infty} \frac{R_n}{B_n} \ge 1+\varepsilon,\quad  \text{or}\quad \liminf_{n\to \infty} \frac{B_n}{R_n} \ge 1+\varepsilon. $$ 
\end{lem}
\begin{proof}
Hypothesis~\ref{hypo} (ii) ensures that there exist $n_0$ and $\varepsilon\in(0,1/2)$, such that for all $n\ge n_0$, 
\begin{equation}\label{cond.iibis}
w\big( \frac{\tau_n}{2} + \sqrt{\sigma_n}\big) \ge (1+8\varepsilon)\cdot w(\frac{\tau_n}{2}),
\end{equation}
and by Hypothesis~\ref{hypo} (i) one can further assume that for all $n\ge n_0$, 
 \begin{equation}\label{cond.ibis}
\varepsilon \sigma_n \ge  \sqrt{\sigma_n}. 
\end{equation}
Note now that since $R_n + B_n = \tau_n$, for all $n\ge 0$, one has $\max(R_{n_0},B_{n_0}) \ge \tau_{n_0} /2$. Suppose to fix ideas that $R_{n_0} \ge \tau_{n_0} /2$. Then by the central limit theorem, 
there exist a constant $c>0$, such that on this event, letting $\mathcal F_n = \sigma(R_k,k\le n)$, 
\begin{equation}\label{stepn0}
\mathbb P\Big(R_{n_0+1} \ge \frac{\tau_{n_0+1}}{2} + \sqrt{\sigma_{n_0+1}}\mid \mathcal F_{n_0}\Big)\ge c.
\end{equation}
Note also that by~\eqref{cond.iibis}, on the event $R_n \ge \frac{\tau_n}{2} + \sqrt{\sigma_n}$, one has using that $w$ is nondecreasing, 
$$\frac{w(R_n)}{w(R_n) + w(B_n)}\ge \frac 12 + \frac{2\varepsilon}{1+4\varepsilon}> \frac 12 + \varepsilon, $$ 
and thus by Chernoff inequality, for some constant $\kappa=\kappa(\varepsilon)>0$, 
\begin{equation}\label{eq.chernoff}
\mathbb P\Big(R_{n+1} - R_n \ge (\frac 12 + \varepsilon ) \sigma_{n+1} \mid \mathcal F_n\Big) \ge 1 - \exp(- \kappa \cdot\sigma_{n+1}).  
\end{equation}
Thanks to~\eqref{cond.ibis}, one can then iterate this estimate, and obtain by induction that on the event $R_{n_0}\ge \tau_{n_0}/2$, with $c$ as in~\eqref{stepn0}, 
\begin{equation}\label{eq:c'}
\mathbb P\Big(R_{n+1}-R_n \ge (\frac 12 + \varepsilon ) \sigma_{n+1}, \forall n \ge n_0+1 \mid \mathcal F_{n_0}\Big)  \ge c \prod_{n\ge n_0+1} \big(1-\exp(- \kappa \cdot\sigma_{n+1})\big). 
\end{equation} 
Using now Hypothesis~\ref{hypo} (i), one can see that the product above is positive. We then define inductively a sequence of stopping times $(s_k)_{k\ge 0}$ as follows. First $s_0 = n_0$, 
and given $s_k$, we let 
\begin{equation*}
s_{k+1} = \left\{ 
\begin{array}{ll}
\inf\{n \ge s_k+1 : R_{n+1} - R_n < (\frac 12 + \varepsilon) \sigma_{n+1}\} & \text{if }R_{s_k}\ge \frac{\tau_{s_k}}{2} \\
 \inf\{n \ge s_k+1 : B_{n+1} - B_n < (\frac 12 + \varepsilon) \sigma_{n+1}\} & \text{else.} 
\end{array}
\right. 
\end{equation*}
By~\eqref{eq:c'}, one can see that for any $k\ge 0$, almost surely on the event $s_k<\infty$, one has 
$$\mathbb P(s_{k+1} = \infty \mid \mathcal F_{s_k}) \ge c':=  c \prod_{n\ge n_0+1} \big(1-\exp(- \kappa \cdot\sigma_{n+1})\big). $$ 
By induction, this gives for any $N\ge 1$,
$$\mathbb P(s_k<\infty, \ \forall k\le N) \le (1- c')^N,$$
and hence letting $N$ go to infinity we get, 
$$\mathbb P(s_k<\infty, \ \forall k\ge 0) = 0. $$ 
In other words, almost surely there exists $k\ge 1$, such that $s_k= \infty$, and the lemma follows.   
\end{proof}

\begin{lem}\label{lem:loc2}
Assume that Hypothesis~\ref{hypo} $(i)$, $(ii)$ and $(iii)$ hold. Then almost surely, 
either $\lim_{n\to \infty} \frac{R_n}{B_n}= \infty$, or $\lim_{n\to \infty}\frac{B_n}{R_n} =  \infty$. 
\end{lem}
\begin{proof}
Let $\varepsilon$ be given by Lemma~\ref{lem:loc1}. By Hypothesis~\ref{hypo} (iii), there exists $\delta>0$, which we can always assume to be smaller than $\varepsilon$, such that for all $\gamma\ge 1+\frac{\varepsilon}{2}$, one has 
$\liminf_{x\to \infty}\frac{w(\gamma x)}{w(x)} \ge \gamma + \delta$. 
We will then show by induction on $k\ge 0$, that for any $k\ge 0$, almost surely, 
\begin{equation}\label{induction}
\liminf_{n\to \infty} \frac{R_n}{B_n}\ge 1+\varepsilon + \frac{k\delta}{2}, \quad \text{or} \quad \liminf_{n\to \infty} \frac{B_n}{R_n}\ge 1+\varepsilon + \frac{k\delta}{2},  
\end{equation}
which will prove the lemma.  The case $k=0$ is precisely given by Lemma~\ref{lem:loc1}, so we assume now that~\eqref{induction} holds for some $k$ and we prove it for $k+1$. Set 
$$\gamma_1 = 1+ \varepsilon + \frac{k\delta}{2}-\frac{\delta}{4}.$$
Since $\delta\le \varepsilon$ by definition, one has $\gamma_1\ge 1+ \frac{\varepsilon}{2}$, and thus there exists $x_0\ge 0$, such that for all $x\ge x_0$, 
$$\frac{w(\gamma_1 x)}{w(x)} \ge \gamma_1 + \delta.$$
Note furthermore that if $ \min(R_\infty,B_\infty)<\infty$, the result is immediate, hence one can assume that $R_\infty = B_\infty = \infty$. 
On this event, and using also the induction hypothesis, we deduce that 
$$n_0 : =  \inf \big\{ n\ge 0 : \max(R_n,B_n) \ge \gamma_1 \min(R_n,B_n) \text{ and } \min(R_n,B_n)  \ge x_0\big\} <\infty. $$ 
Moreover, as in the proof of Lemma~\ref{lem:loc1}, there exists a constant $\kappa=\kappa(\delta,k,\varepsilon)>0$, such that almost surely for any $n\ge n_0$, 
on the event $ R_n\ge \gamma_1 B_n$, 
\begin{equation}\label{induction.accroissements}
\mathbb P\Big(R_{j+1} - R_j \ge \frac{\gamma_1 + \delta-\rho}{\gamma_1 + \delta+1} \cdot \sigma_{j+1}, \ \forall j\ge n\ \Big|\ \mathcal F_n\Big)\ge \prod_{j\ge n} (1- e^{-\kappa \cdot\sigma_{j+1}}), 
\end{equation}
where $\rho$ is a constant to be fixed later, as a function of $\gamma_1$ and $\delta$, and $\mathcal F_n = \sigma(R_k,k\le n)$. 
The above product is larger than the constant $c$ defined by 
$$c:= \prod_{j\ge 0}  (1- e^{-\kappa \cdot\sigma_{j+1}}),$$
which is positive thanks to Hypothesis~\ref{hypo} (i). Then similarly as in the proof of Lemma~\ref{lem:loc1} we define two sequences of stopping times 
$(s_k)_{k\ge 0}$ and $(t_k)_{k\ge 0}$, by 
$s_0 = n_0$, and for $k\ge 0$, 
$$
t_k = \inf\{n\ge s_k : \max(R_n,B_n) \ge \gamma_1 \min(R_n,B_n)\}, 
$$ 
and 
\begin{equation*}
s_{k+1} = \left\{ 
\begin{array}{ll}
\inf\{n \ge t_k : R_{n+1} - R_n < \frac{\gamma_1 + \delta-\rho}{\gamma_1 + \delta+1} \cdot \sigma_{n+1}\} & \text{if }R_{t_k}\ge B_{t_k} \\
 \inf\{n \ge t_k : B_{n+1} - B_n < \frac{\gamma_1 + \delta-\rho}{\gamma_1 + \delta+1} \cdot \sigma_{n+1}\} & \text{else.} 
\end{array}
\right. 
\end{equation*}
By the induction hypothesis, we know that almost surely, for any $k\ge 0$, on the event when $s_k$ is finite, the time $t_k$ is also finite, and by~\eqref{induction.accroissements}, almost surely for any $k\ge 0$, on the event $s_k<\infty$, 
$$\mathbb P(s_{k+1} = \infty \mid \mathcal F_{s_k}) \ge c>0. $$  
As in the proof of Lemma~\ref{lem:loc1}, it follows that almost surely there exists $k\ge 0$, such that $s_k = \infty$, which proves that almost surely, 
$$\liminf_{n\to \infty} \frac{R_n}{B_n}\ge \frac{\gamma_1 + \delta - \rho}{1+ \rho}, \quad \text{or}\quad \liminf_{n\to \infty} \frac{B_n}{R_n}\ge \frac{\gamma_1 + \delta - \rho}{1+ \rho}. $$
One can then choose $\rho$ such that $\frac{\gamma_1 + \delta - \rho}{1+ \rho}\ge \gamma_1 + \frac{3\delta}{4}$, and this concludes the proof of the lemma.  
\end{proof}

We are now in position to prove our first main result in this section. 

\begin{proof}[Proof of Theorem~\ref{theo.GPUexpbis}]
Let $n\ge 0$ be fixed, and recall that $\mathcal F_n = \sigma(R_k,k\le n)$. 
Let 
$$A= \Big\{ \lim_{k\to \infty} \frac{R_k}{B_k} = \infty\Big\} \cup \Big\{\lim_{k\to \infty} \frac{B_k}{R_k} = \infty\Big\}.$$ 
By Lemma~\ref{lem:loc2}, one knows that $\mathbb P(A) = 1$, and hence also almost surely, $\mathbb P(A \mid \mathcal F_n) = 1$.  
This implies that  for any $K\ge 1$, almost surely,
\begin{equation}\label{limiteRinftyBinfty}
\mathbb P(R_\infty = B_\infty = \infty \mid \mathcal F_n) = \lim_{k\to \infty} \mathbb P\big(R_\infty = B_\infty = \infty, \, A_{k,K}\ \big|\ \mathcal F_n\big), 
\end{equation}
where 
$$A_{k,K} = \big\{ R_j \ge K B_j, \ \forall j\ge k \big\}\cup \big\{ B_j \ge K R_j, \ \forall j\ge k \big\}. $$ 
Now given any $k$, conditionally on $R_k$, and say on the event $\{R_k \ge KB_k\}$, we define the process $(R_j,B_j)_{j\ge k}$ using a modification of the continuous-time embedding from Section~\ref{sec:GPU2}. Namely, for the time-line of blue balls, we use the same construction, with a fixed stack of exponential random variables $(\xi_j^B)_{j\ge B_k}$. However, for the time-line of red balls we use different stacks at each step, namely, at step $j$, we use independent mean one exponential random variables $(\xi_{\ell}^{(j)})_{1\le \ell \le \sigma_{j+1}}$. In particular even if one of these clocks is engaged before the end of $j$-th step, but does not ring, we do not use the remaining time for the next step. It is not difficult to see that Lemma~\ref{lem:Rubin} remains valid under this modification of the process (using a similar argument). 
Moreover, by construction, on the event $\{R_\infty = B_\infty = \infty\}$, denoting by $T_\infty$ the total time duration of the process (which we do not need to assume  to be finite here), one has on one hand, almost surely,  
$$T_\infty \ge \sum_{j\ge B_k} \frac{\xi_j^B}{w(j)},$$
and on the other hand, almost surely, 
$$T_\infty \le \sum_{j\ge k} \frac{\sum_{\ell=1}^{\sigma_{j+1}} \xi_\ell^{(j)}}{w(R_j)}. $$ 
Therefore, one has the inclusion of events, with $\varepsilon = 1/K$, 
$$\{R_\infty = B_\infty = \infty\}\cap \{R_j\ge K B_j, \, \forall j\ge k\} \ \subset \left\{\sum_{j\ge k} \frac{\sum_{\ell=1}^{\sigma_{j+1}} \xi_\ell^{(j)}}{w((1-\varepsilon)\tau_j)} \ge 
\sum_{j\ge \varepsilon \tau_k} \frac{\xi_j^B}{w(j)}\right\}.$$
Consequently, letting 
$$Z_k^+(\varepsilon) = \sum_{j\ge k} \frac{\sum_{\ell=1}^{\sigma_{j+1}} \xi_\ell^{(j)}}{w((1-\varepsilon)\tau_j)}, \quad\text{and}\quad Z_k^-(\varepsilon) = \sum_{j\ge \varepsilon \tau_k} \frac{\xi_j^B}{w(j)},$$
one has almost surely for any $k$, on the event $R_k\ge K B_k$, 
$$\mathbb P\Big(R_\infty = B_\infty = \infty, \ R_j \ge K B_j, \ \forall j\ge k \ \Big|\ \mathcal F_k\Big) \le \mathbb P\big(Z_k^+(\varepsilon) \ge Z_k^-(\varepsilon)\big). $$ 
However, by Hypothesis~\ref{hypo} (iv), if $K$ is chosen large enough, one has for all $k$ large enough, 
$$\mathbb E[Z_k^+(\varepsilon)] = \sum_{j\ge k} \frac{\sigma_{j+1}}{w((1-\varepsilon)\tau_j)} \le \frac 14 \cdot \mathbb E[Z_k^-(\varepsilon)]. $$ 
Furthermore, by Markov's inequality, 
$$\mathbb P\big(Z_k^+(\varepsilon) < 2 \mathbb E[Z_k^+(\varepsilon)] \big) \ge 1/2, $$
and observing that 
$$\text{Var}(Z_k^-(\varepsilon)) = \sum_{j\ge \varepsilon \tau_k} \frac{1}{w(j)^2} \le \mathbb E[Z_k^-(\varepsilon)]^2, $$
we deduce using Paley-Zygmund inequality  that 
$$\mathbb P\Big(Z_k^-(\varepsilon) \ge \frac 12 \mathbb E[Z_k^-(\varepsilon) ]\Big) \ge 
\frac 14\cdot  \frac{\mathbb E[Z_k^-(\varepsilon)]^2}{\mathbb E[Z_k^-(\varepsilon) ]^2 +\text{Var}(Z_k^-(\varepsilon))} \ge \frac 1{8}. $$  
Altogether, using independence between $Z_k^+(\varepsilon)$ and $Z_k^-(\varepsilon)$, this yields for any $k$ large enough, 
$$\mathbb P\big(Z_k^+(\varepsilon) \ge Z_k^-(\varepsilon)\big) \le 1 - \frac 1{16}. $$ 
We deduce that for any $k$ large enough, almost surely on the event $R_k\ge K B_k$, 
$$\mathbb P\Big(R_\infty = B_\infty = \infty, \ R_j \ge K B_j, \ \forall j\ge k \ \Big|\ \mathcal F_k\Big) \le 1- \frac 1{16}. $$ 
By symmetry a similar result holds on the event $B_k \ge KR_k$, and thus almost surely 
$$\mathbb P\big(R_\infty = B_\infty = \infty, \, A_{k,K} \ \big|\ \mathcal F_k\big) \le 1- \frac 1{16}. $$ 
Together with~\eqref{limiteRinftyBinfty}, we get that almost surely, for any $n\ge 0$, 
$$\mathbb P(R_\infty = B_\infty = \infty \mid \mathcal F_n)  \le 1 - \frac 1{16}. $$ 
Letting $n$ go to infinity, and using martingale convergence results, we get that almost surely $R_\infty$ or $B_\infty$ is finite, as wanted. 
\end{proof}

We can now give a proof of the result highlighted in the introduction about the specific choice of weight function growing stretched exponentially fast, and drawing sequence as a power law. As we will see the result is not a direct application of Theorem~\ref{theo.GPUexpbis}, at least not for all choices of the parameters. 

\begin{proof}[Proof of Theorem~\ref{thm:GPUexp}]
We assume here that $w(n) = \exp(cn^\alpha)$, and $\sigma_n = n^\beta$, with $\beta> \frac{1-\alpha}{\alpha}$. 
Note that in this case Hypotheses~\ref{hypo} $(i)$, $(iii)$ and $(iv)$ are satisfied, but not necessarily $(ii)$. However, the latter condition was only used for proving Lemma~\ref{lem:loc1}, and in fact a careful look at the rest of the proof reveals that all one needs is to show that for some $\varepsilon>0$, almost surely 
\begin{equation}\label{newgoal}
\limsup_{n\to \infty} \frac{\max(R_n,B_n)}{\min(R_n,B_n)} \ge 1+\varepsilon. 
\end{equation} 
Indeed, this is the only input which is used in the induction argument in the proof of Lemma~\ref{lem:loc2}. 
Now as shown in the proof of Lemma~\ref{lem:loc1}, by the central limit theorem, one knows that almost surely infinitely often $\max(R_n,B_n) \ge \frac{\tau_n}{2} +\sqrt{\sigma_n}$. 
 Moreover, if say $R_n \ge \frac{\tau_n}{2} + \sqrt{\sigma_n}$, then letting $\varepsilon_n = \sqrt{\sigma_n} \cdot \tau_n^{\alpha-1}$, one has for $n$ large enough, 
$$ \frac{w(R_n)}{w(B_n) + w(R_n)} \ge \frac{1+\varepsilon_n}{2+\varepsilon_n}\ge \frac 12 + \frac{\min(\varepsilon_n, 1)}{6},$$ and thus basic properties of Binomial random variables ensure that there exists $c>0$, such that for all $n$ large enough, almost surely, on the event $R_n \ge \frac{\tau_n}{2} + \sqrt{\sigma_n}$, 
$$\mathbb P\big(R_{n+1}\ge \frac{\tau_{n+1}}2 +  \sigma_{n+1}\cdot \frac{\min(\varepsilon_n,1)}{6} \mid \mathcal F_n\big)\ge c. $$  
By a standard application of the conditional Borel-Cantelli lemma, this shows that almost surely infinitely often, 
$$R_{n+1} \ge  \frac{\tau_{n+1}}2 +  \sigma_{n+1}\cdot \frac{\min(\varepsilon_n,1)}{6}.$$

Note also that $\sigma_{n+1}\cdot \varepsilon_n \gtrsim n^\gamma\cdot \sqrt{\sigma_{n+1}}$, with $\gamma =  \beta -(1+\beta)(1- \alpha) > 0$, where the last inequality is thanks to the hypothesis $\beta> \frac{1-\alpha}{\alpha}$. Repeating the same argument as above a finite number of times, we can see that if $k\ge 1$ is the smallest integer such that 
$k\gamma + \beta/2 \ge \beta$, one has for some constant $c'>0$, almost surely, infinitely often,  
$$R_{n+k+1} \ge  \frac{\tau_{n+k+1}}{2} + c'\sigma_{n+k+1}. $$
Now using again the fact that $\beta> \frac{1-\alpha}{\alpha}$, it can be observed that for some constant $\varepsilon>0$, 
$$\liminf_{n\to \infty} \frac{w(\frac{\tau_n}{2} + c'\sigma_n) }{w(\frac{\tau_n}{2} ) }\ge 1+ \varepsilon.$$ 
Then we can follow the same argument as in the proof of Lemma~\ref{lem:loc1}, and deduce~\eqref{newgoal}. 
With this in hand, as explained above, the same proof as for Theorem~\ref{theo.GPUexpbis} applies, which concludes the proof of Theorem~\ref{thm:GPUexp}. 
\end{proof}

To conclude we present another set of hypotheses, which is slightly different from Hypothesis~\ref{hypo}, neither stronger nor weaker. The main reason for introducing them, is that they can be used to prove that some VRBRW restricted to three sites localizes almost surely on two sites, while it would not have been possible to prove this using the previous hypotheses; see the next section and Proposition~\ref{prop:appVRBRW}.  

\begin{Hypo}\label{hypo2}
\begin{itemize}
\item[$(i)$] The sequence $(\sigma_n)_{n\ge 1}$ is unbounded, or in other words  
$$\limsup_{n\to \infty} \sigma_n = \infty. $$  
\item[$(ii)$] $$\liminf_{n\to \infty} \frac{w( \frac{\tau_n}{2} + \sqrt{\sigma_n})}{w(\frac{\tau_n}{2})} >1. $$ 
\item[$(iii)$]  For any $\varepsilon>0$, there exists $n_0$, such that for all $n\ge n_0$,  
$$\sum_{i=n}^\infty \frac{\sigma_{i+1}}{w((\frac 12 +\varepsilon)\tau_i)} \le \frac 14 \sum_{i = \lfloor (\frac 12 - \varepsilon) \tau_n\rfloor }^\infty \frac 1{w(i)}. $$ 
\end{itemize}
\end{Hypo}
\begin{rem}\emph{
Note that $(i)$ is not restrictive, since the case of bounded sequences $(\sigma_n)_{n\ge 1}$ is handled by Theorem~\ref{thm:GPU}. Hypothesis $(ii)$ is the same as before, and $(iii)$ is stronger than Hypothesis~\ref{hypo} $(iv)$, since the condition must now be fulfilled for all $\varepsilon>0$. }
\end{rem}
\begin{theo}\label{thm.loc.bis}
Assume that Hypothesis~\ref{hypo2} $(i)-(iii)$ hold. Then 
$$\mathbb P(R_\infty<\infty \text{ or } B_\infty<\infty)= 1. $$ 
\end{theo}
\begin{proof}
The proof is similar to the proof of Theorem~\ref{thm:GPU}. We just need to be slightly more careful in the first step. As in the proof of Lemma~\ref{lem:loc1}, let us first 
consider $\varepsilon>0$, and $n_0\ge 1$, such that 
$$\frac{w( \frac{\tau_n}{2} + \sqrt{\sigma_n})}{w(\frac{\tau_n}{2})} \ge 1+ 8 \varepsilon, $$
for all $n\ge n_0$, and we assume without loss of generality that $\varepsilon\le 1/8$. Note that their existence is guaranteed by Hypothesis~\ref{hypo2} $(ii)$.  We then define two sequences of stopping times $(s_n)_{n\ge 0}$ and 
$(t_n)_{n\ge 0}$, by $s_0 = n_0$, and inductively, for all $n\ge 0$, 
$$t_n = \inf \{ k \ge s_n : \varepsilon\sigma_{k+2} \ge \sqrt{\sigma_{k+2}}\}, $$
and 
\begin{equation*}
s_{n+1} = \left\{ 
\begin{array}{ll} 
\inf \{n \ge t_n +2: R_n - R_{t_n+1} <(\frac 12 + \varepsilon)(\tau_n- \tau_{t_n+1}) \} & \text{if }R_{t_n} \ge B_{t_n} \\
\inf \{n \ge t_n +2: B_n - B_{t_n+1} <(\frac 12 + \varepsilon)(\tau_n- \tau_{t_n+1})\} & \text{else}. 
\end{array}
\right. 
\end{equation*}
Note that by Hypothesis~\ref{hypo2} $(i)$, for any $n\ge 0$, on the event when $s_n$ is finite, $t_n$ is also almost surely finite. 
Now assume that $s_n$, hence $t_n$ as well, is finite for some $n$. Assume also to fix ideas that $R_{t_n} \ge B_{t_n}$. Then using the central limit theorem, we know that for some constant $c_1>0$, almost surely, 
\begin{equation}\label{CLT}
\mathbb P\big(R_{t_n +1} \ge \frac{\tau_{t_n+1}}{2} + \sqrt{\sigma_{t_n +1}} \mid \mathcal F_{t_n}\big)\ge c_1. 
\end{equation}
Moreover, since by construction $t_n+1\ge n_0$, on the event in the probability above, one has, 
$$\frac{w(R_{t_n+1})}{w(R_{t_n+1}) + w(B_{t_n+1})} \ge \frac 12 + \frac{2 \varepsilon}{1+ 4\varepsilon}\ge \frac 12 + \frac 43 \varepsilon,$$
using our assumption that $\varepsilon\le 1/8$, for the last inequality. Now Chernoff's bound guarantees that for some constant $c_2>0$, on the event in the probability in~\eqref{CLT}, almost surely, 
$$ \mathbb P\big(R_{t_n +2} \ge \frac{\tau_{t_n+2}}{2} + \varepsilon \sigma_{t_n+2}  \mid \mathcal F_{t_n+1}\big)\ge c_2.$$  
Note that by definition of $t_n$, one has $\varepsilon \sigma_{t_n+2} \ge \sqrt{\sigma_{t_n+2}}$, and in fact also for any $k\ge 0$, 
$$\varepsilon (\sigma_{t_n+2} + k) \ge \max(\sqrt{\sigma_{t_n+2}},\sqrt{k}). $$  
Hence as long as $R_i \ge \frac{\tau_i}{2} + \varepsilon(\tau_i - \tau_{t_n+2})$, for $ i\ge t_n +2$, the process $(R_i)_{i\ge t_n +2}$ dominates the positions of a simple random walk with drift $(4/3)\varepsilon$ to the right, at the times $\tau_i - \tau_{t_n+2}$. In other words, letting $(S_j)_{j\ge 0}$ be the Markov process starting from $S_0 = R_{t_n +2}$, jumping only to nearest neighbors, and from $x$ to $x+1$ with probability $1/2 + (4/3)\varepsilon$, for all $x$, one can couple it with our urn process in a way that  
$$ \left\{ S_j - S_0\ge \varepsilon j,\,  \forall j\ge 0\right\} \subseteq \left\{ R_i \ge \frac{\tau_i}{2} + \varepsilon(\tau_i - \tau_{t_n+2}), \, \forall i\ge t_n+2 \right\}. $$  
Now it is well known that the event on the left above has a positive probability to happen, say $c_3>0$. Consequently, and to summarize, we just have shown that almost surely, for any $n\ge 0$,  
$$\mathbb P\big( s_{n+1} = \infty \mid \mathcal F_{s_n}\big) \ge c_1c_2c_3>0. $$ 
It follows that almost surely there exists $n\ge 0$, such that $s_n = \infty$, which also proves that almost surely, 
$$\liminf_{n\to \infty} \frac{R_n}{B_n} \ge 1+ \varepsilon, \quad \text{or} \quad \liminf_{n\to \infty} \frac{B_n}{R_n} \ge 1+ \varepsilon. $$ 
We conclude the proof of the theorem, exactly as for Theorem~\ref{theo.GPUexpbis}.
\end{proof}
%%%%%%%%%%%%%%%%%%%%%%%%%%%%%%

\subsection{Application to the VRBRW on three sites}\label{sec:VRBRW3sites} 
Here we investigate some consequences of the previous results for the VRBRW restricted to three sites. 
\begin{prop} \label{prop:VRBRW3sites}
Consider a VRBRW on $\{-1,0,1\}$. Then,  
$$\sum_{k\ge 1} \frac 1{w(k)}<\infty\ \Longleftrightarrow\  \mathbb P(Z_\infty(1)<\infty)>0.$$   
\end{prop}

\begin{proof}
Note that by identifying the local times at $-1$ and $1$ with the number of red and blue balls respectively, one can see that the VRBRW on $\{-1,0,1\}$ is equivalent to a GPU process, where for any $n\ge 1$, the number of draws at step $n$ is random and equal to  $\mathcal Z_{2n-1}$, and where initially there are no ball of each color.  Moreover, the fact that $w$ is nondecreasing ensures that for any sequence $(\sigma_n)_{n\ge 1}$, letting $\tau_n = \sigma_1+\dots + \sigma_n$, one has 
$$\sum_{n\ge 1}\frac{\sigma_n}{w(\tau_n)} \le \sum_{k=1}^\infty \frac 1{w(k)}. $$ 
Hence, if $\sum_{k\ge 0} \frac 1{w(k)}<\infty$, then using also Lemma~\ref{lem:esp}, we get 
$$\mathbb E\Big[ \sum_{n\ge 1}\frac{\mathcal Z_{2n+1}}{w(\mathcal Z_1+\mathcal Z_3+ \dots + \mathcal Z_{2n-1} )}\Big] =
\mathbb E\Big[ \sum_{n\ge 1}\frac{\mathcal Z_{2n-1}+2 \sigma^2}{w(\mathcal Z_1+\mathcal Z_3+ \dots + \mathcal Z_{2n-1} )}\Big]   <\infty,$$ 
which entails that almost surely, 
$$\sum_{n\ge 1}\frac{\mathcal Z_{2n+1}}{w(\mathcal Z_1+\mathcal Z_3+ \dots + \mathcal Z_{2n-1} )}<\infty. $$ 
Then Proposition~\ref{prop:GPU} (i) implies that almost surely, conditionally on $\mathcal T_\infty$, the probability that only one color is drawn infinitely often is positive. Integrating with respect to $\mathcal T_\infty$ shows that $\mathbb P(Z_\infty(-1)<\infty) >0$.

Assume now that $\sum_{k\ge 0} \frac 1{w(k)}= \infty$. Then note that for any sequence $(\sigma_n)_{n\ge 0}$, 
$$\sum_{n\ge 1}\frac{\sigma_n}{w(\tau_{n-1})} \ge \sum_{k\ge \sigma_1} \frac 1{w(k)}  = \infty,$$ 
and thus applying this time Proposition~\ref{prop:GPU} (ii) gives that almost surely, conditionally on $\mathcal T_\infty$, both colors are picked infinitely often with probability one, which after integration with respect to $\mathcal T_\infty$ gives that $\mathbb P(Z_\infty(-1) = Z_\infty(1) = \infty) = 1$. 
\end{proof}

In the next result, we consider a very specific choice of weight function. Our goal is not to be exhaustive, but simply 
to exhibit an example where almost sure localization on 
two sites can be proved for a VRBRW restricted to three sites. However, even for this specific example, the extension of the result to the model of the VRBRW on $\mathbb Z$ seems challenging. 
\begin{prop}\label{prop:appVRBRW}
Assume that $w(n) = \exp(cn)$, for some constant $c>0$, and consider a VRBRW on $\{-1,0,1\}$. Then  
$$\mathbb  P(Z_\infty(-1)= Z_\infty(1) = \infty)= 0. $$ 
\end{prop}
\begin{proof}
We just need to prove that all the conditions of Hypothesis~\ref{hypo2} are satisfied almost surely with $\sigma_n = \mathcal Z_{2n-1}$. Item $(i)$ follows from the fact that $(\mathcal Z_{2n-1})_{n\ge 1}$ is an irreducible Markov chain on (a possibly sub-semigroup of) the set of nonnegative integers. Item $(ii)$ is always true for weights of the form $w(n) = \exp(cn)$. Hence it just amounts to prove $(iii)$ now. Fix $\varepsilon>0$. For $k\ge 1$, let $\tau_k = \mathcal Z_1 + \mathcal Z_3 + \dots + \mathcal Z_{2k-1}$, and $\mathcal F_k = \sigma(R_1,\dots,R_k,\tau_1,\dots,\tau_k)$. By Lemma~\ref{lem:esp}, and Markov's inequality, one has almost surely 
$$\sum_{k\ge 1} \mathbb P(\mathcal Z_{2k+1} \ge k^2 \mathcal Z_{2k-1}\mid \mathcal F_k) < \infty. $$ 
Hence by the conditional Borel-Cantelli lemma, almost surely, $\mathcal Z_{2k+1} \le k^2 \mathcal Z_{2k-1}$, for all $k$ large enough. 
Now by definition, one has $\tau_n \ge n$, and thus almost surely for $n$ large enough, 
 \begin{align*}
\sum_{k\ge n} \frac{\mathcal Z_{2k+1}}{w((\frac 12 + \varepsilon)\tau_k)} & \le \sum_{k\ge n} \frac{k^2\mathcal Z_{2k-1}}{w((\frac 12 + \varepsilon)\tau_k)} 
\le \sum_{k\ge n} \frac{\mathcal Z_{2k-1}}{\exp(c\frac{\tau_k}{2})} 
 \le  \sum_{i\ge \tau_{n-1}} \exp(-ci/2)  \\
 & = \frac{1}{1- e^{-c/2}}\cdot \exp\big(-c\frac{ \tau_{n-1}}{2}\big).   
\end{align*}
On the other hand (forgetting the integer parts), 
$$\sum_{i\ge (\frac 12 - \varepsilon)\tau_n} \frac 1{w(i)} = \frac 1{1- e^{-c}} \exp\big(-c(\frac 12- \varepsilon)\tau_n\big). $$
Then the result follows from Lemma~\ref{lem:ratiolim} and again the fact that $\tau_n\ge n$.  
\end{proof}

%%%%%%%%%%%%%%%%%%%%%%%%%%%%%%%%%%%%%%%%%%%%%%%%%%%%%%%%%%%%%%%%%%%%%%%%

\section{Proof of Proposition~\ref{prop.strong}}\label{sec:propstrong}
We assume here that $\sum_{n\ge 0} \frac 1{w(n)}<\infty$. Then we first prove that the process almost surely visits a finite number of sites. To this end, let us consider $(\xi^{(i)}_n(x))_{i\ge 1, n\ge 0, x\in \mathbb Z}$ and 
$(\widetilde \xi_n(x))_{n\ge 0,x\in \mathbb Z}$  two sequences of independent random variables with law $\mu$ and $\widetilde \mu$ respectively, representing the number of children of the particles present at site $x$ at time $n$, either normal for the former sequence or special for the latter. Let also $\Delta Z_n(x) = Z_n(x) - Z_{n-1}(x)$ be the number of particles present at site $x$ at time $n$. 
Observe that each time the process visits a new site $x\ge 1$, say at a time $\tau_1\ge 1$, one has 
$$Z_{\tau_1}(x) \le \widetilde \xi_{\tau_1-1}(x-1) + \sum_{i=1}^{\Delta Z_{\tau_1-1}(x-1)} \xi^{(i)}_{\tau_1-1}(x-1),$$
since there is at most one special vertex per generation.  
Similarly, on the event that $x+1$ has not yet been visited by the $n$-th time, say $\tau_n$, at which $x$ is visited by at least one particle, the number of particles that jump to $x$ at that time, $\Delta Z_{\tau_n}(x)$, satisfies 
$$\Delta Z_{\tau_n}(x) \le \widetilde \xi_{\tau_n-1}(x-1) + \sum_{i=1}^{\Delta Z_{\tau_n - 1}(x-1)} \xi^{(i)}_{\tau_n-1}(x-1). $$ 
Then conditionally on the past up to time $\tau_n$, and on the event that $x+1$ has not yet been visited by that time, the probability that it remains not visited at time $\tau_n+1$ equals 
\begin{align*}
\mathbb P(Z_{\tau_n +1}(x+1) = 0\mid \mathcal F_{\tau_n})&  = \mathbb E\Big[\Big(1- \frac{w(0)}{w(0) + w(Z_{\tau_n-1}(x-1))}\Big)^{\Delta Z_{\tau_n}(x)} \mid \mathcal F_{\tau_n}\Big] \\
& \ge  1 - \frac{w(0)\cdot \mathbb E\Big[\Delta Z_{\tau_n}(x) \mid \mathcal F_{\tau_n}\Big]}{w(0) + w(Z_{\tau_n-1}(x-1))}\\ 
& \ge  1 - \frac{w(0)\cdot (\Delta Z_{\tau_n-1}(x-1)+ 1+\sigma^2) }{w(0) + w(Z_{\tau_n-1}(x-1))},
\end{align*}
using that the mean of the size biased distribution is $1+\sigma^2$. Now observe that for any sequence $(\sigma_n)_{n\ge 1}$ and $(\tau_n)_{n\ge 0}$ for which $\sigma_n \le \tau_n - \tau_{n-1}$ for all $n\ge 1$, one has 
$$\sum_{n\ge 1} \frac{\sigma_n}{w(\tau_n)} \le \sum_{k=0}^\infty \frac 1{w(k)}<\infty. $$ 
Hence almost surely, conditionally on the past before the first visit to site $x\ge 1$, the probability that $x+1$ is never visited is uniformly bounded from below by a positive constant. It follows that almost surely $\sup \mathcal R'<\infty$. By symmetry we also have that almost surely $\inf \mathcal R'>-\infty$, whence the fact that almost surely $|\mathcal R'|$ is finite.

Now, regarding the localization on two sites with positive probability, one can use the same arguments as in the proofs of Propositions~\ref{prop:VRBRW3sites} and~\ref{prop:GPU} (i), which show that 
$\mathbb P(\mathcal R = \{0,1\})>0$. 
This concludes the proof of the proposition. 

\vspace{0.2cm}
\textbf{Acknowledgements:} Work supported by the grant ANR-22-CE40-0012 (project LOCAL).

\end{document}